\documentclass[12pt, a4paper]{amsart}
\usepackage{amsmath,amssymb,amscd,amsfonts}
\newtheorem{theorem}{Theorem}[section]
\newtheorem{rem}[theorem]{Remark}

\newtheorem{lemma}[theorem]{Lemma}

\newtheorem{proposition}[theorem]{Proposition}
\newtheorem{conjecture}[theorem]{Conjecture}
\newtheorem{corollary}[theorem]{Corollary}

\newcommand\Q{{\mathbb{Q}}}

\def\cO{{\mathcal O}}

\def\cH{{\mathcal H}}
\def\cF{{\mathcal F}}

\def\cC{{\mathcal C}}

\begin{document}
\title[]{Characteristic foliation on non-uniruled smooth divisors on 
hyperk\"ahler manifolds}

\author{Ekaterina Amerik, Fr\'ed\'eric Campana}

\address{Institut Elie Cartan \\
Universit\'e Henri Poincar\'e\\
B. P. 70239, F-54506 Vandoeuvre-l\`es-Nancy Cedex, France\\
and: Institut Universitaire de France}
\email{frederic.campana@univ-lorraine.fr}

\address{National Research University Higher School of Economics\\
Department of Mathematics\\
Laboratory of Algebraic Geometry\\
Vavilova 7, 117332 Moscow, Russia}

\email{ekaterina.amerik@gmail.com}

\begin{abstract} 
We prove that the characteristic foliation $\cF$ on a non-singular divisor $D$ in an 
irreducible projective hyperk\"ahler manifold $X$ cannot be algebraic, unless 
the leaves of $\cF$ are rational curves or $X$ is a surface. More generally, 
we show that
if $X$ is an arbitrary projective manifold carrying a holomorphic symplectic $2$-form, and $D$ and $\cF$
are as above, then $\cF$ can be algebraic with non-rational leaves only when, 
up to a finite \'etale cover, $X$ is the product of a symplectic projective 
manifold $Y$ with a symplectic surface and $D$ is the pull-back of a curve on this 
surface.  

When $D$ is of general type, the fact that $\cF$ cannot be algebraic unless
$X$ is a surface 
was proved by Hwang and Viehweg. The main new ingredient for our results is 
the observation that the canonical class of 
the (orbifold) base of the family of leaves is zero. This implies, 
in particular, the isotriviality of the family of leaves of $\cF$.

\

 R\'ESUM\'E: Nous montrons que si le feuilletage caract\'eristique $\cF$ d'un diviseur $D$ lisse d'une vari\'et\'e projective complexe symplectique irr\'eductible $X$ est alg\'ebrique, alors ou bien $X$ est une surface, ou bien les feuilles de $\cF$ sont des courbes rationnelles. Lorsque $D$ est de type g\'en\'eral, ce r\'esultat est d\^u \`a Hwang et Viehweg. Nous en d\'eduisons, lorsque $X$ est une vari\'et\'e projective complexe arbitraire munie d'une $2$-forme symplectique holomorphe, et $D$, $\cF$ comme ci-dessus, que si les feuilles $X$ sont des courbes alg\'ebriques non-rationnelles, alors, apr\`es rev\^etement \'etale fini, $X$ est le produit d'une surface $K3$ ou ab\'elienne $S$ par une vari\'et\'e symplectique $Y$, et $D=C\times Y$ pour une courbe $C\subset S$.

L'ingr\'edient principal nouveau de la d\'emonstration est l'observation que la classe canonique (orbifolde) de la base de la famille des feuilles est triviale. Ceci implique, en particulier, l'isotrivialit\'e de la famille des feuilles de $\cF$.

\end{abstract}

\maketitle

\section{Introduction}

Let $X$ be a projective manifold equipped with a holomorphic symplectic
form $\sigma$. Let $D$ be
a smooth divisor on $X$. At each point of $D$, the restriction of $\sigma$
to $D$ has one-dimensional kernel. This gives a non-singular foliation $\cF$ on 
$D$, called {\it the characteristic foliation}. We say that $\cF$ is {\it algebraic} if all its leaves 
are compact complex curves.

If $D$ is uniruled, the characteristic foliation $\cF$ is always algebraic. 
Indeed, its leaves are the
fibres of the rational quotient fibration on $D$ (see for example \cite{AV}, Section 4). 
On the other hand, J.-M. Hwang and E. Viehweg proved in \cite{HV} that $\cF$ cannot be
algebraic when $D$ is of general 
type, except for the trivial case when $dim(X)=2$.
The aim of this article is to classify the examples where $\cF$ is algebraic 
and $D$ is not uniruled.

Our main result is as follows.

\begin{theorem}\label{isotriviality} Let $X$ be a projective manifold with a holomorphic 
symplectic form $\sigma$ and let $D$ be a smooth hypersurface 
in $X$. If $\cF$ as above is algebraic and the genus of its general leaf is $g>0$, 
then the associated fibration is isotrivial and $K_D$ is nef and abundant, with
$\nu(K_D)=\kappa(D)=1$ when $g\geq 2$ and $\nu(K_D)=\kappa(D)=0$ when $g=1$.
\end{theorem}

Here $\nu$ denotes the numerical dimension and $\kappa$ the Kodaira dimension.
In general, $\kappa(D)$ does not exceed $\nu(K_D)$, 
and $K_D$ is said to be abundant when the two dimensions coincide 
(by a 
result of Kawamata, this implies the
semiampleness of $K_D$, so this notion is important in the minimal model program).

What we actually are going to prove is a slightly more general result.
Consider a smooth projective variety $D$ of dimension $d$ carrying a nowhere vanishing holomorphic 
$(d-1)$-form $\omega$. Such a form has one-dimensional kernel at each point and therefore defines
a smooth rank-one foliation ${\cF}$. Alternatively, the foliations arising in this way are 
those defined by the subbundles of $T_D$ isomorphic to the anticanonical bundle of $D$.
In this situation, we have the following

\begin{theorem}\label{isotriviality-form} If $\cF$ is algebraic, 
then the associated fibration $f:D\to B$ is isotrivial without multiple fibers in codimension one and 
the canonical class $K_B$ is trivial.
\end{theorem}

We refer to subsection 2.1 for the definition and discussion of the fibration
associated to a smooth algebraic foliation of rank one.

When $D$ is a divisor in a holomorphic symplectic manifold $(X,\sigma)$ of dimension $d+1=2n$, 
one recovers the first part of  
theorem \ref{isotriviality} by taking the form $\sigma^{\wedge{(n-1)}}$ for $\omega$, since the
kernel of $\sigma$ is then equal to that of $\omega$; the assertions on the numerical and Kodaira
dimension are deduced from theorem \ref{isotriviality-form} in a standard way.

The next two theorems are consequences of theorem \ref{isotriviality}.

\begin{theorem}\label{irreducibleHK} Let $X$, $D$, $\cF$ be as in theorem 
\ref{isotriviality}, and 
suppose moreover
that $X$ is irreducible (that is, simply connected and with $h^{2,0}(X)=1$).
If $\cF$ is algebraic and $D$ is not uniruled, then $dim(X)=2$.
\end{theorem}

By the Bogomolov decomposition theorem, up to a finite \'etale covering, any compact K\"ahler 
symplectic manifold is a product of a torus and several irreducible holomorphic symplectic
manifolds. Since our assumptions on $D$ and $\cF$ are preserved under finite \'etale coverings,
theorem \ref{irreducibleHK} is valid for holomorphic symplectic manifolds with $h^{2,0}=1$ and finite 
fundamental group. Moreover we may consider only the case of such products in the sequel.

\begin{rem} The smoothness assumption is essential, as one sees by considering the Hilbert square $X$ of an elliptic $K3$-surface $g:S\to \Bbb P^1$: one has a fibration $h:X\to \Bbb P^2=Sym^2(\Bbb P^1)$. If $C\subset \Bbb P^2$ is the ramification conic of the natural 
$2$-cyclic cover $(\Bbb P^1)^2\to \Bbb P^2$, and $L\subset \Bbb P^2$ is a line tangent to $C$, then the characteristic foliation on 
the singular divisor $D:=h^{-1}(L)$ is algebraic with $g=1$. One obtains  
similar examples with $g>1$ by considering the image of $C\times S$ in the Hilbert square
of $S$, where $S$ is an arbitrary $K3$ surface and $C\subset S$ is a curve. 
\end{rem}

\begin{theorem}\label{classification} Let $X$, $D$, $\cF$ be as in theorem \ref{isotriviality}. 
Suppose that $D$ is 
non-uniruled and $\cF$ is algebraic. Then, possibly after a finite \'etale
covering, $X =S\times Y$, where $dim(S)=2$, both $S$ and $Y$ are complex 
projective manifolds 
carrying holomorphic
symplectic forms $\sigma_S$, $\sigma_Y$, and $D = C\times Y$, where
$C\subset S$ is a curve. 
\end{theorem}

\begin{rem}\label{form}

The surface $S$ from theorem \ref{classification} is, up to a finite cover, either $K3$ or abelian. 
In the first case, $\sigma =p^*\sigma_S\oplus q^*\sigma_Y$ on $TX\cong p^*TS\oplus q^*TY$ (where
$p$, $q$ denote the projections) by K\"unneth formula.
 
In the second case, one still has $\sigma =p^*\sigma_S\oplus q^*\sigma_Y$ when $g>1$. Indeed, by 
K\"unneth formula (and Bogomolov decomposition) one reduces to the case when $Y$ is also an abelian variety, 
and the decomposition then 
follows by a straightforward linear-algebraic computation.
 
In contrast to these cases, when $S$ is an abelian surface and $g=1$, $\sigma$
is not always a direct sum  (see example \ref{notsum}).
\end{rem}

The main ideas of the proof of Theorem
\ref{isotriviality-form} are as follows. Suppose that $D$ is not uniruled
and that $\cF$ is algebraic. Then $\cF$ defines a holomorphic fibration
$f: D\to B$ such that its non-singular fibers are curves of genus $g>0$,
and the singular fibers are multiple curves with smooth reduction. The base has only quotient singularities
by Reeb stability. We prove that the codimension
of the locus of multiple fibers in $D$ (and of its image in $B$) is at least two. Therefore
the form $\omega$ descends to $B$ outside of a codimension-two locus; this trivializes the canonical class
of $B$.

The generic semi-positivity theorem of \cite{CP13}, in the simpler case when there is no orbifold structure, 
now implies that the Iitaka dimension of the determinant of any 
subsheaf of the cotangent sheaf of $B$ is non-positive. On the other
hand, Hwang and Viehweg construct such a subsheaf (coming from the Kodaira-Spencer map) 
with Iitaka dimension equal to the number of 
moduli
of the fibres of $f$. Therefore the family $f$ must be isotrivial.

As an application, we deduce in section 
\ref{lagrang}
a certain case of the Lagrangian conjecture 
on a projective (and, more generally, compact K\"ahler) 
irreducible holomorphic symplectic manifold of dimension $2n$
from the Abundance conjecture in dimension $2n-1$. We therefore solve this case 
unconditionally for $n=2$, since the Abundance conjecture is known for 
threefolds (\cite{K'}; see also \cite{CHP} for the generalization to the K\"ahler case). This was our initial 
motivation for this research. When the 
research has 
been completed,
Chenyang Xu has informed us that for projective manifolds, this case of the Lagrangian 
conjecture
follows from a fundamental result of Demailly, Hacon and Paun 
(\cite{DHP}). As no algebraic proof of \cite{DHP} is known, our result also gives a 
simple algebro-geometric alternative for hyperk\"ahler manifolds (see section 
\ref{lagrang} for statements and proofs). 

The next section is devoted to the proof of theorems \ref{isotriviality} and \ref{isotriviality-form}.
The following two prove theorems \ref{irreducibleHK} and \ref{classification} respectively.
In the last section, we treat our application to the Lagrangian conjecture.


\section{Some numerical invariants of the characteristic foliation}

\subsection{Smooth rank $1$ foliations} \label{sf}

Let $D$ be a $d$-dimensional ($d\geq 2$) connected K\"ahler manifold carrying a non-singular holomorphic foliation $\cF$ of rank $1$.
The foliation $\cF$ is called algebraic when all its leaves are compact complex curves. 
A non-singular algebraic foliation induces a proper holomorphic map $f:D\to {\cC}(D)$ to a component ${\cC}(D)$
of the cycle space of $D$. Indeed,
the general leaves of $\cF$ are smooth curves varying in a dominating family of cycles on $D$; by compactness of ${\cC}(D)$, 
one has well-defined limit cycles which must be supported on the special leaves, and the
multiplicity of such a cycle is uniquely determined by pairing with the K\"ahler class. Taking the normalization of the image 
if necessary, we obtain a proper holomorphic map  $f:D\to B$ onto a $(d-1)$-dimensional normal base $B$. 
It is well-known that in such a situation,
the holonomy groups of the leaves are finite (this amounts to the boundedness of the volume of the leaves which holds in the K\"ahler case,
see for example \cite{E}). Therefore by Reeb stability (see \cite{MM}, or else \cite{HV} which develops the construction
of \cite{MM} in the holomorphic case in some detail), locally in some saturated neighbourhood of each leaf $C$ of holonomy group $G_C$, 
our foliation is the quotient 
of $T\times \tilde{C}$,
where $T$ is a local transverse and $\tilde{C}$ is $G_C$-covering of $C$,  
by the natural action of $G_C$.

In particular, $B$ has only quotient singularities and so is $\Q$-factorial, and $f$ is ``uni-smooth'', that is, the reduction of any of its fibres is a 
smooth projective curve.

Let $g$ denote the genus of a non-singular fiber 
of $f$. If $g=0$, the holonomy groups are trivial and all fibres of $f$ are smooth reduced rational curves, $B$ is smooth, 
$f$ submersive.
If $g>0$, $f$ may have multiple fibres, of genus one when $g=1$ and of genus
greater than one (but possibly smaller than $g$) when $g>1$.

If $g=1$ and $B$ is compact, it is well-known that $f$ must be {\bf isotrivial}: indeed, the $j$-function then holomorphically maps $B$ to $\Bbb C$. In fact the holomorphicity of $j$ near the multiple fibres is easily checked: from Reeb stability we obtain the local boundedness of $j$, and then use the normality of $B$.

A pair $D,\cF$ as above arises, for example, when $D$ is a smooth connected divisor in a 
$2n$-dimensional projective (or compact K\"ahler) manifold $X$ carrying a holomorphic symplectic 
$2$-form $\sigma$. The foliation $\cF$ is then given, at each $x\in D$, as the $\sigma$-orthogonal to $TD_x$ at $x$. In this case $d=2n-1$. In general, $\cF$ will not be algebraic. One particular case when $\cF$ is algebraic
is that of a uniruled $D$: the leaves of $\cF$ are then precisely the fibres 
of the rational quotient fibration of $D$ (see for instance \cite{AV}, 
section 4), so $g=0$.

We will elucidate below the situation when $\cF$ 
is algebraic and $g>0$. 

Note that in this example, the quotient bundle $T_D/\cF$ carries a symplectic form, so it has trivial 
determinant.
Therefore the line bundle $\cF$ is isomorphic to the anticanonical bundle of $D$ (by adjunction,
this is ${\cO}_D(-D)$).

The purpose of this section is to prove the following result, which is stated as theorem 
\ref{isotriviality} in the introduction.

\begin{theorem}\label{tmr} Let $D$ be a smooth divisor in a projective holomorphic symplectic variety $(X, \sigma)$, and $\cF$ the foliation on $D$ given by the kernel of $\sigma|_D$. If $D$ is non-uniruled and $\cF$ is algebraic, then
the corresponding fibration $f:D\to B$ is isotrivial, $K_D$ is nef and abundant, $\nu(K_D)=\kappa(D)=1$ if $g\geq 2$, and $\nu(K_D)=\kappa(D)=0$ if $g=1$. 
\end{theorem}

This shall be a consequence of a more general isotriviality result stated as theorem \ref{isotriviality-form}:

\begin{theorem}\label{tmr-form} Let $D$ be a complex projective manifold of dimension $d$ carrying a nowhere vanishing holomorphic
$(d-1)$-form $\omega$. Let $\cF$ be the foliation defined as the kernel of $\omega$. Suppose $\cF$ is algebraic. Then the corresponding
fibration $f:D\to B$ is isotrivial and submersive in codimension two, and the canonical class of $B$ is trivial. 
\end{theorem}

Our first idea is to introduce the {\it orbifold base} of a fibration and to show that the orbifold structure is actually trivial
when the fibration is defined by a non-vanishing holomorphic $(d-1)$-form.

\subsection{Orbifold base}\label{ob}

Let $f:D\to B$ be a uni-smooth fibration in curves, with $D$ smooth K\"ahler and $B$ $\Q$-factorial.
We define (as in \cite{Ca04}, in a much more general situation there) 
the {\it orbifold base} $(B,\Delta)$ for $f$ as follows: for each irreducible reduced Weil ($\Q$-Cartier) divisor $E\subset B$, set $E'=f^{-1}(E)$. This is an irreducible divisor, 
and $f^*(E)=m_f(E)E'$ for some positive integer $m_f(E)$. This integer is equal to $1$ for all but finitely many $E$. Set $\Delta=\sum_{E\subset B} (1-\frac{1}{m_f(E)})E$. The divisor $\Delta$ thus carries the information about the multiple fibers of $f$ in codimension one,
but the coefficients of $\Delta$ are ``orbifold multiplicities'' varying between zero and one rather than the multiplicities of the 
fibers. Over a neighbourhood of a general point $b\in \Delta$ - that is, a point outside of $Sing(B)$ and $Sing(\Delta)$ -
the map $f$ is locally given by $(z_1, \dots, z_{d-1}, w)\mapsto (z_1^m, \dots z_{d-1})=(u_1,\dots, u_{d-1})$, where $m=m_f(E)$
for the component $E$ of $\Delta$ which contains $b$.

\begin{lemma}\label{nodelta} Suppose that $f$ is given by the kernel of a non-vanishing holomorphic $(d-1)$-form $\omega$. Then $f$ has no 
multiple fibers in codimension one, that is, $\Delta=0$, and $K_B$ is trivial.
\end{lemma}

\begin{proof} The question on multiple fibers is local on $B$, and $B$ is smooth in codimension one. 
We can thus assume that $B$ is a polydisc in $\Bbb C^{d-1}$, with coordinates 
$(u_1,\dots, u_{d-1})=(u,u')$ ($u$ being the first coordinate and $u'$ the 
$(d-2)$-tuple of others), and that $f$ has multiple fibres of multiplicity $m>1$ over the divisor $E$ defined by the equation 
$u=0$. Since the form $\omega$ is d-closed, and its kernel is $Ker(df)^{sat}$, the saturation being taken in $TD$, 
it descends over $B-E$ to a holomorphic $(d-1)$-form $\alpha$ on $B-E$ such that $\omega=f^*(\alpha)$ on $f^{-1}(B-E)$
(see e.g. \cite{S}, lemma 6, where the full argument is given for holomorphic symplectic forms; it immediately generalizes
to our setting).

We are going to show that $\alpha$ extends holomorphically to $B$, and that $m=1$. Write 
$\alpha=G(u, u')du\wedge du'$ (where $du'$ stands for the wedge product of $du_i$ for $i>1$). We claim that $\vert G(u,u')\vert =e^g.\vert u\vert^{-c}$, with $c=1-\frac{1}{m}$, where $g$ is a real-valued bounded function, after possibly shrinking $B$ near $(0,0)$.

Let, indeed, $B'\subset D$ be a smooth local multisection of degree $m$ over $B$ meeting transversally the reduction of the fibre of $f$ over $(0,0)\in B$. We can choose the coordinates $(z_1,z_2,\dots z_{d-1},w)=(z, z', w)$ on $D$ near the intersection point $(0,0,0)$ of $B'$ and the fibre $D_{(0,0)}$ of $f$ over $(0,0)$ in such a way that $f(z,z',w)=(z^m,z')$, and $B'$ is defined by the equation $y=0$. Restricting $\omega$ to $B'$, we see that $f^*(\alpha)=G(z^m,z').m.z^{m-1}.dz\wedge dz'=\omega|_{B'}=h(z,z')dz\wedge dz'$, for some nowhere vanishing function $h(z,z')=H(z^m,z')=H(u,u')$, whenever $u=z^m\neq 0$.

Thus $\vert G(u,u')\vert=\vert G(z^m,z')\vert=\frac{\vert H(u,u')\vert}{m}.\frac{1}{\vert u\vert ^c}=e^{g(u,u')}.\frac{1}{\vert u\vert^c}$.

The following well-known fact now shows that $\alpha$ extends holomorphically to $B$, and hence $c$ must be zero
and $m=1$ as claimed. 

Let $G(u,u')$ be a holomorphic function defined on $B-E$, where $B$ is a polydisc centered at $(0,0)$ in $\Bbb C^{n-1}$, and $E$ is the divisor defined by $u=0$ in $B$. Assume that, for some $\varepsilon >0$, $\vert G(u,u')\vert \leq C.\vert u\vert ^{-(1-\varepsilon)}$ for some positive constant $C$ independent on $u'$. Then $G(u,u')$ extends holomorphically across the divisor $u=0$.

Indeed, fix $u'$, the Laurent expansion $G(u,u')=\sum_{k=-\infty}^{k=+\infty} a_k(u').u^k$ of $G$ has then coefficients $a_k(u')r^k=\frac{1}{2\pi}\int_{0}^{2\pi}e^{-ikt}G(re^{it},u')dt$ (cf. Henri Cartan, Th\'eorie \'el\'ementaire des fonctions analytiques, Hermann 1961, p. 86, formula (2.1)). The bound on $\vert G\vert$ implies that $\vert a_k(u')\vert \leq Cr^{-k-1+\varepsilon}$ for $0<r<<1$. This implies that $a_k(u')=0$ if $k<0$, by letting $r\to 0^+$.

It remains to show that the Weil ($\Q$-Cartier) divisor $K_B$ is trivial. Indeed the form $\omega$ descends to a non-vanishing holomorphic
form on the complement of a codimension-two subset of $B$. Hence the triviality of $K_B$.   
\end{proof}

\begin{rem}\label{byreeb} Another way to see this is by Reeb stability. Indeed in a neighbourhood of a multiple fiber $C$ over a general
point of $\Delta$, which has cyclic holonomy of order $m$,
$\omega$ must lift as a $G_C$-invariant form to the $G_C$-covering coming from Reeb stability, but this is impossible by the explicit
local computation.
\end{rem}

\begin{rem}\label{example-mf} The map
$f:D\to B$ given by a global nonvanishing $(d-1)$-form may have multiple fibers
in codimension two: take for instance $D=(E\times E\times C)/G$ where $E$ is an
elliptic curve, $C$ is a curve equipped with a fixed-point-free involution
and $G$ a group of order two where the non-trivial element acts as $-Id$ on
$E\times E$ and as that involution on $C$. Then the projection onto
the quotient of $E\times E$ by $-Id$ has isolated multiple fibers, and is given
by the kernel of a $2$-form which is the exterior product of $1$-forms on $E$.
\end{rem}

\subsection{Isotriviality of the fibration}\label{cdm}

As we have already remarked, the isotriviality of the family of curves $f:D\to B$ associated to $\cF$ is clear when $g=0$ or
$g=1$, so we assume in this section that $g\geq 2$. All varieties are assumed to be projective (or quasiprojective,
when we work outside of a suitable codimension two subset such as $Sing(B)$).
Define the sheaf $\Omega^1_B$ as the direct image $j_*\Omega^1_{B^{sm}}$ where $j:B^{sm}\to B$ is
the embedding of the smooth part of $B$ in $B$. The following theorem is a direct consequence of the strenghtening of Miyaoka's generic
semi-positivity theorem (\cite{Mi}, see \cite{PM} p. 66-67, Theorem 2.14, 2.15 for a formulation
adapted to our purposes) given in lemma \ref{gsp} below.


\begin{theorem}\label{odiff} Let $B$ be a normal projective variety with log-canonical singularities such that $K_B\equiv 0$. 
Let $L$ be a coherent rank-one subsheaf of $(\Omega^1_B)^{\otimes k}$ for some $k>0$. Then $deg_{C}(L_{\vert C})\leq 0$ for a sufficiently general complete intersection curve $C$ cut out on $B$ by members of a linear system $|lH|$, $l>>0$, where $H$ is an ample 
line bundle on $B$. 

In particular, 
for any integer $m>0$ one has: $h^0(B^{sm},L^{\otimes m})\leq 1$ and so $\kappa(B^{sm},det(\cF))\leq 0$, for any coherent subsheaf $\cF\subset \Omega^1_B$.
\end{theorem}

\begin{proof} By lemma \ref{gsp} below, the quotient $Q:=((\Omega^1_B)^{\otimes k}/L)$ restricted to $C$ has non-negative degree. Thus the degree of the (locally free) sheaf $L|_C$ is non-positive, since $deg_C((\Omega^1_B)^{\otimes k})=0$. \end{proof}

The following lemma is a special case of the main result of \cite{CP13} when $\Delta=0$. Its proof is a considerably simplified version of the general case, in particular, no use of 
orbifold differentials is required. We refer to \cite{CP13} for details (see \cite{CP15} for a more general result; the reader may also consult \cite{Cl} which
exposes the main ideas and techniques of both \cite{CP13} and \cite{CP15}).

\begin{lemma}\label{gsp} Let $B$ be an $n$-dimensional, normal, connected, projective variety with log-canonical singularities. Assume that $K_B$ is pseudo-effective. Then, for any $m>0$, any quotient $Q$ of $(\Omega^1_B)^{\otimes m}$ has non-negative slope with respect to any ample polarisation $\alpha:=H^{n-1}$ of $B$.
\end{lemma}

\begin{proof} By general properties of slopes (see e.g. pages 9-10 of \cite{Cl}), it suffices to show that the 
minimal slope $\mu_{\alpha}^{min}(\Omega^1_B) \geq 0$ (by the {\it minimal slope}, one means the smallest possible 
slope of a quotient sheaf with respect to $\alpha$); recall that
we are interested in degrees of restrictions to curves $C$ not passing through the singularities of 
$B$ and therefore all sheaves we consider are locally free on such $C$. Assume there is a
quotient of $\Omega^1_B$ with $\alpha$-negative slope, then the maximal destabilising subsheaf for the dual 
defines an $\alpha$-semistable foliation 
$\cF$ on $B$ with $\alpha$-positive slope (integrability follows from Miyaoka's slope argument). 
By \cite{BMQ}, $\cF$ is algebraic. So there exists a rational fibration $g:B\dasharrow Z$ such that $\cF=Ker(dg)$. 

Taking a neat model (see \cite{CP13}, p. 848) of $g$, obtained by blowing up $B$ and $Z$, we get 
$g':B'\to Z'$ with $B',\ Z'$ smooth. We can write 
$K_{B'}+\Delta'=b^*(K_B)+E$, where $b:B'\to B$ is our blow-up, for $\Delta'$ and $E$ some 
effective 
$b$-exceptional $\Bbb Q$-divisors on $B'$ without common components.  
The divisor $\Delta'$ on $B'$ is an orbifold divisor (that is has coefficients between $0$ and $1$) 
because $B$ is log-canonical, and $K_{B'}+\Delta'$ is pseudo-effective, since so is $K_B$. Let
$D(g',0)$ denote the ramification divisor of $g'$, that is, the sum of the $(g'^*F-(g'^*F)_{red})^{surj}$
over all prime divisors $F$ of $Z'$,
where the superscript means that we consider only the non-exceptional components of $g'^*F-(g'^*F)_{red}$,
i.e. those which map surjectively to $F$ (cf. \cite{CP13}, p.848). 
Theorem 2.11 of \cite{CP13} shows that $M:=K_{B'/Z'}+(\Delta')^{hor}-D(g',0)$ is 
pseudo-effective as well (here the superscript denotes the ``horizontal part'', dominating the base). Thus, denoting by $C'$ the strict transform of a curve $C$ which is
a generic complete intersection of large multiples of $H$, we obtain  $MC'\geq 0$. On the other 
hand, it follows from Proposition 1.9 of \cite{CP13} that up to a positive normalisation constant $k$, we have:
$-det(\cF)C=k.MC'$, a contradiction since by construction $det(\cF)C>0$. Note that though 
the calculation
of this proposition is made in the orbifold context we actually do not need orbifold differentials: as
the generic curve $C$ avoids the singularities, $C'$ avoids $\Delta'$, so that only the
ramification of $g'$ contributes.
\end{proof}

\begin{rem}\label{rmiyvscp} An example where Miyaoka's result does not apply, while lemma \ref{gsp} does, is the following Ueno surface. Let $A=E\times E$ be the product of two copies of the elliptic curve $E$ with complex multiplication by $i=\sqrt{-1}$. Let $S:=A/\Bbb Z_4$, the generator acting by $i$ simultaneously on both factors. Then $S$ is a rational surface with 16 quotient singularities, not all canonical. On $S$ there is no pair $(\cF,H)$ consisting of a rank $1$ foliation $\cF$ and a polarisation $H$ such that the $H$-slope of $\cF$ is positive. This follows from \ref{gsp}, but can also be checked directly. Indeed, otherwise both $\cF$ and $H$ could be lifted to $A$ with this same intersection property, since the quotient map $q:A\to S$ is \'etale over $S^{sm}$, and $A$ is smooth. But such a foliation does not exist on $A$, since its tangent bundle is trivial. The surface $S$ is uniruled, therefore the absence of such foliations is not directly implied by \cite{Mi}.
\end{rem}

We return to the proof of Theorem 2.2.

The determinant of any subsheaf of $\Omega^1_B$ restricted to $B^{sm}$ has non-positive 
Kodaira dimension; this also remains true for
finite coverings of $B$, \'etale over $B^{sm}$.

Following \cite{HV}, we now construct a subsheaf of $\Omega^1_B$ (or more precisely of $\Omega^1_{B'}$
where $B'$ is such a covering) such that the Kodaira dimension of its
determinant over $B^{sm}$ is equal to the variation of moduli of our family of curves; the argument is shorter here
since we have remarked that $f$ is submersive in codimension one.

Indeed, it suffices to do so outside of a codimension-two algebraic subset in 
$B$, that is, over $B^0$ which is smooth and such that the restriction $f:D^0 \to B^0$ of
$f:D\to B$ is a smooth family of curves. It is well-known (see e.g. \cite{HV}, Lemma 3.1) that,
after replacing $B^0$ by a finite \'etale covering, the family $f:D^0\to B^0$ becomes the pull-back of the universal
family of curves with level $N$ structure $g:{\cC}_g^{[N]}\to M_g^{[N]}$ under a morphism 
$j: B^0\to M_g^{[N]}$ for a suitable $N>>0$.

Since $D^0$ is now a smooth family of curves over a smooth base $B^0$, 
one can consider the ``Kodaira-Spencer map'' $$f_*(\omega_{D^0/B^0}^{\otimes 2})\to \Omega^1_{B^0}$$ 
obtained by dualizing the usual Kodaira-Spencer map from $T_{B^0}$ to 
$R^1f_*T_{D^0/B^0}$ associated to the family of curves $f:D^0\to B^0$. 
Let $\cH \subset \Omega^1_{B^0}$ be its image: it is a coherent subsheaf of $\Omega^1_{B^0}$. 
Moreover, it is functorial in $B^0$, that is, its construction commutes with base change.

\begin{proposition}\label{hv} (cf. \cite{HV}, proposition 4.4) Assume that $g\geq 2$. Then $\kappa(B^0,det(\cH))=Var(f)=dim(Im(j))$.
\end{proposition}

\begin{proof} The sheaf $f_*(\omega_{D^0/B^0}^{\otimes 2})$ is the pull-back by $j$ of $g_*(\omega_{{\cC}_g^{[N]}/M_g^{[N]}}^{\otimes 2})$,
and the latter is ample by \cite{HV}, Proposition 4.3. We conclude by \cite{HV}, Lemma 4.2.
\end{proof}

\begin{corollary}\label{isotrivial}
The fibration $f:D\to B$ is isotrivial.
\end{corollary}

Indeed, by \ref{odiff} we know that the Kodaira dimension of the determinant 
of any subsheaf of $\Omega^1_{B^{sm}}$ is non-positive, and so 
$\kappa(B^{sm},det(\cH))=Var(f)=0$. This finishes the proof of theorem \ref{tmr-form}.

\subsection{A more general conjectural isotriviality statement.}

The corollary \ref{isotrivial} is a special case of the following more general conjectural statement, which slightly generalises \cite{T}\footnote{In \cite{T}, the conjecture is established when $B$ is smooth and $\Delta=0$.}:

\begin{conjecture} Let $f:X\to B$ be a proper, connected, quasi-smooth\footnote{That is, the reduction of every fibre is smooth.} fibration of quasi-projective varieties, where $X$ is smooth and $B$ is normal. Assume that the (reduced) fibres of $f$ have semi-ample canonical class, and that the orbifold base $(B,\Delta)$ of $f$ is special in the following sense (cf. \cite{Ca07}): for any $p>0$ and any coherent rank-one subsheaf $L\subset (f^*(\Omega^p_B))^{sat}$, where the saturation takes place 
in $\Omega^p_X$, one has $\kappa(X,L)<p$. Then $f$ is isotrivial.
\end{conjecture}


We would like to remark that the special case of this conjecture when $f:X\to B$ is a family of curves and the orbifold canonical bundle of the base is trivial can be proved by an argument similar to the one just given but much more subtle, using the orbifold generic semi-positivity 
of 
\cite{CP13}  and the full argument of \cite{HV}. Since this turns out to be irrelevant for the characteristic foliation by 
lemma \ref{nodelta}, we intend to publish this elsewhere.

\subsection{Consequences of isotriviality.}\label{isot}

\

Our goal now is to get the information on $K_D$ once the isotriviality is established. All arguments work in the compact K\"ahler case.
 Let us first remark that the relative canonical divisor
$K_{D/B}$ is well-defined as a $\Q$-Cartier divisor, and $K_D\equiv K_{D/B}$
since $K_B$ is trivial.

We first make a normalized base-change to remove all multiple fibers.

\begin{lemma}\label{lisok} Let $D$ be a compact connected K\"ahler manifold with a 
smooth rank-one foliation $\cF$ with compact leaves of genus $g\geq 1$. Let $f:D\to B$ be the associated proper fibration. Consider the normalized base-change 
$f_D:(D\times_B D)^{\nu}\to D$. Then $f_D$ is smooth.
\end{lemma}

\begin{proof} By definition of a foliation, a neighbourhood of $x\in D$ is 
isomorphic to $U'\times F$,
where $F$ is a small open subset of the leaf through $x$ and $U'$ is a local transverse to the foliation.
Moreover, by Reeb stability, a small neighbourhood $U$ of $b\in B$ is $U'/G$ where $G$ is
the holonomy group, and $D_U'=(D\times_U U')^{\nu}$ is smooth over $U'$ and \'etale over 
$D_U=f^{-1}(U)$.
Hence $(D\times_BD)^{\nu}$, which locally in a neighbourhood of $x$ is naturally isomorphic to
$(D\times_U (U'\times F))^{\nu}=D_U'\times F$, is smooth over $D$: indeed the projection to $D$
is, locally, the composition of the smooth projection to $D_U'$ with the natural \'etale projection from $D_U'$ to $D_U$.
\end{proof}

Denote by $f':D'\to B'$ our new smooth family (so that $B'=D$ and $D'=(D\times_B D)^{\nu}$) and by $s:D'\to D$ the natural projection. 
Notice that since the normalization procedure only concerns the codimension-two locus, we have $K_{D'/B'}\equiv s^*K_{D/B}$.

It is well-known that a smooth isotrivial family of curves of genus $g$,
after a suitable finite base change,
becomes a product when $g\geq 2$, and a principal fibre bundle when $g=1$.
More precisely, we have the following lemma: 

\begin{lemma}\label{numdim} There exists a finite proper map $h':B''\to B'$ such that after 
base-changing $f'$ by $h'$, we get $f'':D''\to B''$ and $s':D''\to D'$ with the
following properties: 
$D''\cong F\times B''$ over $B''$ when $g\geq 2$, and $f'':D''\to B''$ is a principal fibre bundle if $g=1$. Moreover, $K_{D''/B''}$ 
is nef, 
$\kappa(D'',K_{D''/B''})=\nu(D'',K_{D''/B''})=1$ if $g\geq 2$, and $\kappa(D'',K_{D''/B''})=\nu(D'',K_{D''/B''})=0$ if $g=1.$
\end{lemma}

Here $\nu$ denotes the numerical dimension.

\begin{proof} The smooth isotrivial family $f'$ is a locally trivial bundle with structure
group $Aut(F)$, where $F$ is a fiber.  If $g\geq 2$, this is a finite group, so that
the bundle trivializes after a finite covering $h':B''\to B'$. If $g=1$, we get the 
 principal bundle structure after a finite covering corresponding to the quotient 
of $Aut(F)$ by the translation subgroup. The second claim is obvious when $g\geq 2$. 
When $g=1$, we remark 
that $K_{D''/B''}$ is dual to $f''^*(R^1f''_*(\cO_{D''}))$, and the latter
is trivial since translations on an elliptic curve operate trivially on 
cohomology.
\end{proof}

\begin{corollary}\label{liso} Let $f:D\to B$ be as above. Then $K_D$ is nef, 
$\kappa(D)=\nu (D,K_D)=1$ if  $g\geq 2$, and $\kappa(D)=\nu (D,K_D)=0$ if  $g=1.$
\end{corollary}

\begin{proof} Since $$NK_{D''/B''}\equiv Ns'^*(K_{D'/B'})\equiv Ns'^*(s^*(K_{D/B}))\equiv N(h\circ h')^*(K_D),$$ this follows from the preceding lemma,  by the preservation of nefness, numerical dimension and Kodaira-Moishezon dimension under inverse images.\end{proof}

This finishes the proof of the theorem \ref{isotriviality} in the projective case. Remark that when $g=1$, 
this argument also proves the K\"ahler case, since the isotriviality, for which the projectivity assumption
was needed, is then automatic. This shall be used in the proof of Corollary 5.2.

In the next section, we shall give a proof of the theorem \ref{irreducibleHK}.

\section{Divisors on irreducible hyperk\"ahler manifolds.}\label{cor1}

We suppose now that $X$ is a projective irreducible holomorphic symplectic manifold
of dimension $2n\geq 4$,
$D\subset X$ is a smooth non-uniruled divisor on $X$ and the fibres of 
$f: D\to B$ are curves of non-zero genus tangent to the kernel of the restriction of the holomorphic
symplectic form $\sigma$ to $D$. Recall that on the second cohomology of $X$
there is a non-degenerate bilinear form $q$, the {\it Beauville-Bogomolov form}.

By corollary \ref{liso} $\nu(K_D)\leq 1< \frac{dim(X)}{2}$. On the other hand,
we have the following well-known lemma (see for instance \cite{mat-numdim}, lemma 1, keeping in
mind that by Fujiki formula $D^{2n}$ is proportional to 
$q(D,D)^n$ with non-zero coefficient, and that the numerical dimension $\nu(D)$ of a nef divisor $D$ is the maximal
number $k$ such that the cycle $D^k$ is numerically non-trivial).

\begin{lemma}\label{nu}Let $D$ be a non-zero nef divisor on an irreducible hyperk\"ahler manifold $X$. Then either $\nu(D)=dim(D)$ (if $q(D,D)>0$), or $\nu(D)=\frac{dim(X)}{2}$ (if $q(D,D)=0$).
\end{lemma}

Note that $\nu(X, D)=\nu(D, K_D)+1$, since $K_D=D|_D$. Therefore $\nu(D)\leq 2$ 
and the only 
possibility is $dim(X)=4$, $\nu(X, D)=2$, $\nu(D, K_D)=\kappa(D)=1$, $g\geq 2$.
This case can be excluded as follows: since $\kappa(D)=\nu(D, K_D)$, $D$ is a good
minimal model and the Iitaka
fibration $\phi: D\to C$ is a regular map. Its fibers $S$ are equivalent
to $D^2$ as cycles on $X$, and therefore are lagrangian. Indeed, it follows 
from the definition of the Beauville-Bogomolov form $\sigma$ on $X$ that
$$\int_S\sigma\bar{\sigma}=q(D,D)=0,$$
and this implies that the restriction of $\sigma$ to $S$ is zero. 
So the leaves of the characteristic foliation must be contained in the
fibers of $\phi$, giving the fibration of $S$ in curves of genus at least $2$. But
this is impossible on $S$, since $S$ is a minimal surface of Kodaira 
dimension zero. 

This proves theorem \ref{irreducibleHK}.

\section{Divisors on general projective symplectic manifolds.} \label{cor2}

The purpose of this section is to prove theorem \ref{classification}.

Recall the setting: $(X, \sigma)$ is a holomorphic symplectic projective variety, 
$D\subset X$ is a smooth hypersurface such that its characteristic foliation 
$\cF$ is algebraic and the genus $g$ of the leaves is strictly positive. 
We wish to prove that up to a finite \'etale covering, $X$ is a product
with a surface and $D$ is the inverse image of a curve under projection to
this surface.   

By Bogomolov decomposition theorem, we may assume that $X$ is the product of a 
torus $T$ and several irreducible hyperk\"ahler manifolds $H_j$ with $q(H_j)=0$ 
 (here $q$ denotes the irregularity $h^{1,0}$) and $h^{2,0}(H_j)=1$.

\

We distinguish two cases: 

\

{\bf First case: $X$ is not a torus.} We shall proceed by induction on the number of non-torus factors in the Bogomolov decomposition of $X$.

Since $X$ is not a torus, there is an irreducible hyperk\"ahler factor $H$
in the Bogomolov decomposition. If $X=H$, we are done. Otherwise, 
write $X=H\times Y$, where $Y$ is the product of the remaining factors. By K\"unneth formula, we have $\sigma_X=\sigma_H\oplus \sigma_Y$ on 
$TX\cong TH\oplus TY$, since $q(H)=0$. For $y\in Y$ general, 
let $D_y=D\cap (H\times \{y\})$. If this is empty, then $D=H\times D_Y$ for some divisor $D_Y$ of $Y$, which is smooth with algebraic characteristic foliation.
Indeed,
at any point of $D$ the $\sigma_X$-orthogonal to $TD$ is contained in the 
$\sigma_X$-orthogonal  to $TH\subset TD$, whereas $TH^{\perp}=TY$ since 
$\sigma_X$ is a direct sum. We conclude by induction in this case.

Therefore we may suppose that $D$ dominates $Y$. For $y\in Y$ generic, $D_y$ is a smooth non-uniruled divisor on $H\times y$. 
At any point $(h,y)\in D$ such that $D_y\neq H\times y$ is smooth at $h$, 
we have $TD_y=TD\cap TH$. Moreover, at such a point $TH\not\subset TD$ and 
thus, taking the $\sigma$-orthogonals, $\cF\not\subset TY$. We get $(TD_y)^{\perp}=TD^{\perp}\oplus TH^{\perp}=\cF\oplus TY$.  

Since $\sigma$ is a direct sum, the $\sigma_H$-orthogonal of $TD_y$ in $TH$
is the projection of $\cF$ to $TH$.

In other words: the characteristic foliation $\cF_{D_y}$ of $D_y$ inside $H$ is the projection on $TH$ of the characteristic foliation $\cF\subset TX$ along $D_y$. The leaves of $\cF_{D_y}$ are thus the \'etale $p_H$-projections of the leaves of $\cF$ along $D_y$, and so $\cF_{D_y}$ is  algebraic, with non-uniruled leaves. From theorem \ref{irreducibleHK}, we deduce that $H$ is a $K3$-surface, and the 
divisors $D_y$ are curves of genus $g>0$ for $y\in Y$ generic.

When $D_y$ is singular at $h$, one has $TH\subset TD$ at $(h,y)$, and therefore
at such points $\cF\subset TY$.

Fix any $h\in H$ and let $C_y$ denote the leaf of the characteristic foliation 
of $D$
through $(h,y)$. 
By isotriviality, all the curves $C_y$ are isomorphic to each other. When
$y$ varies in the fibre of $D$ over $h$, we thus have a positive-dimensional 
family of nonconstant maps $p_H: C_y\to H$ parameterized by a compact 
(but possibly not connected) variety $D^h$, and all images pass 
through the point 
$h\in H$. After a base-change $\alpha: Z\to D^h$ (not necessarily finite, but 
with $Z$ still compact) of the family of the
leaves, we have a map $p: C_y\times Z\to H$ mapping a section $c\times Z$ to
a point. By the rigidity lemma, all images $p_H(C_y)$ coincide when $y$
varies in a connected component of $Z$; therefore there is only a finite
number of curves $C_y$ through any $h\in H$. By the same reason, such a curve
(that is, the projection of a leaf of $\cF$ to $H$) 
does not intersect its small deformations in the family of the projections
of leaves. The family of such curves is thus at most a one-parameter family,
and there are only finitely many of them through any given point of $H$.

We are thus left with two cases: either all leaves of $\cF$ project
to the same curve on $H$, so that $p_H(D)=C\subset H$ is a curve and we are 
finished; or $p_H(D)=H$. In this last case, $H$ is covered by a one-parameter
family of curves $C_t$, which we may suppose irreducible, such that 
$C_t$ does not intersect its small deformations and there is only a finite
number of $C_t$ through a given point.

Notice also that these $C_t$ have to coincide with the connected components of the 
divisors $D_y$ and therefore the
generic $C_t$ is smooth. By adjunction formula, it is an elliptic curve and
$H$ is fibered in curves $C_t$.

We claim that every $C_t$ is non-singular. Indeed, suppose that some $C_t$
is singular at $h\in H$. It has to be a connected component of a $D_y$ for some
$(h,y)$ on a leaf of $\cF$ projecting to $C_t$. As we have remarked above, the
singularity of $D_y$ at $h$ means that $TH\subset TD$ and therefore $\cF\subset TY$
along a connected component of $p_H^{-1}(h)$. But such a component is 
of strictly positive dimension and therefore would contain a leaf of $\cF$.
So there are at least two leaves of $\cF$ through $(h,y)$, one projecting 
to $C_t$ and another to a point, which is absurd.

Since $H$ is a $K3$-surface, it does not admit an elliptic fibration without 
singular
fibers by topological reasons (non-vanishing of the Euler number). This is the contradiction excluding $p_H(D)=H$, and thus establishing theorem \ref{classification} when $X$ is not a torus. $\square$

\

{\bf Second case: $X=T$ is a torus.} We shall use Ueno's structure theorem for subvarieties of tori (\cite{U'}, Theorem 10.9). 

If $g>1$, then $\kappa(D)=1$. By Ueno's theorem there is a subtorus $K$ of codimension $2$ such that $D$ is the inverse
image of a curve on the quotient: $D=p^{-1}(C)$, where $p:T\to S:=T/K$ is the projection and $C\subset S$ is a curve 
of genus $g'>1$ on the abelian surface $S$. The $\sigma$-orthogonal space to $K$ gives canonically a two-dimensional linear
foliation $\cF_T$ on $T$, such that the intersections of its leaves with $D$ are the leaves of $\cF$, 
hence smooth compact curves which project in an \'etale way by $p$ onto $C$.

Let us show that the leaves of $\cF_T$ are compact. Take a leaf $C$ of $\cF$ through a point $x\in T$. 
It is contained in the leaf $L$ of $\cF_T$ through $x$. Choose a group structure on 
$T$
in such a way that $x=0$. The translate of $C$ by any point $a\in C$ is still
contained in the leaf $L$ since $L$ is linear; on the other hand, it is not equal
to $C$ for $a$ outside of a finite set, since $g(C)>1$. Since $L$ is two-dimensional
and contains a family of compact curves parameterized by a compact base, $L$ must
itself be compact. 

Therefore the leaves of $\cF_T$ are translates of an abelian surface $S'$. It 
suffices now to take a finite \'etale base-change from $S$ to $S'$ to get the
desired form $T'=K\times S'$, $D'=K\times C$, $\sigma$ direct sum of symplectic 
forms on $S'$ and $K$.

If $g=1$, then $\kappa(D)=0$, and $D$ is a subtorus of codimension $1$ with an elliptic fibration. There thus exists an elliptic curve $C\subset T$ and a quotient $\pi: T\to R=T/C$ such that $D=\pi^{-1}(V)$, where $V$ is a codimension $1$ subtorus of the torus $R$. Project $\rho: R\to R/V$, and consider the composition $p:T\to S:=R/V$. Then $S$ is an abelian surface, and $C':=p(C)$ is an elliptic curve on it. Moreover, $D=p^{-1}(C)$. Let $K$ be the kernel of $p$: this is a 
subtorus of $T$ of codimension $2$.
By Poincar\'e reducibility, there exists an abelian surface $S' \subset T$ such that $(S'\cap K)$ is finite. After a finite \'etale cover, $T'=S'\times K$, and $D'=C\times K$ is of the claimed form. 
$\square$

\begin{rem}\label{notsum} In this last case, $\sigma_T$ is in general not the direct sum of symplectic forms on $S'$ and $K$. Take for example $T=S\times A$, 
$D=E\times A$, for $S,A,E\subset S$ Abelian varieties of dimensions $2, (n-2), 1$ respectively, with linear coordinates $(x,y)$ on $S$, $(z_1,...,z_{n-2})$ on $A$, and 
$E$ given by $x=0$. Take $\sigma_S:=dx\wedge dy$, $\sigma_A$ arbitrary on $A$, and $\sigma=\sigma_S+\sigma_A+dx\wedge dz$, for any nonzero linear form $z$ 
on $TA$.
\end{rem}

\section{Application to the Lagrangian conjecture.}\label{lagrang}

Our aim is corollary \ref{csa} below. First we prove the following proposition.

\begin{proposition}\label{alb} Let $D\subset X$ be a smooth hypersurface in a connected compact K\"ahler manifold $X$ of dimension $2n$, carrying a holomorphic symplectic $2$-form $\sigma$. Denote by $\cF$ the characteristic foliation on $D$ defined by $\sigma$. Assume that 
$D$ admits a holomorphic fibration $\psi:D\to S$ onto an $(n-1)$-dimensional connected complex manifold $S$, such that its general fibre is a
lagrangian subvariety of $X$ of zero Kodaira dimension. Then

1. The foliation $\cF$ is $\psi$-vertical (ie: tangent to the fibres of $\psi$). 

2. Either the smooth fibres of $\psi$ are tori, and then $\psi$ is the restriction to $D$ of a holomorphic Lagrangian fibration $\psi'$ on some open neighborhood of $D$ in $X$; or their irregularity $q(F)$ is equal
to $n-1$. In this case the Albanese map $a_F:F\to Alb(F)$ is surjective and connected, and its fibres are elliptic curves which are the leaves of $\cF$. 
Moreover $F$ has a finite \'etale covering which is a torus.
\end{proposition}

\begin{proof} The first claim is obvious, since, at any generic $x\in D$, the $\sigma$-orthogonal to $TD_x$ is included into the $\sigma$-orthogonal to $TF_x$
(where $F$ denotes the fibre of $\psi$ through $x$), which is equal to itself since $F$ is Lagrangian.

Since the deformations
of our Lagrangian fibres $F$ cover $D$, we have $q(F)=h^0(Y, \Omega^1_X)=h^0(Y,N_{Y/X})\geq dim(D)-dim(F)=n-1$. Note that $q(F)\leq n$, since
the Albanese map of a variety with zero Kodaira dimension is surjective with connected fibres by 
\cite{K}.

If $q(F)=n$, $F$ is bimeromorphic to a torus. Since it admits an  everywhere regular foliation, it must
be a torus. In this case $F$ deforms in an $n$-dimensional family and this
gives a fibration of a neighbourhood of $F$ in $X$ (indeed, the normal
bundle to $F$ in $X$ is trivial since it is isomorphic to the cotangent bundle
by the lagrangian condition).  Otherwise, $q(F)=n-1$ and the fibres of the Albanese map $a_F$ are one-dimensional.
In fact these are elliptic curves by $C_{n,n-1}$ (\cite{Vi}), and this also implies that 
$F$ has a finite \'etale covering which is a torus.

Finally, the leaves of $\cF$ inside $F$ are tangent to the fibres of $a_F$. Indeed, since $q=n-1$ and $F$ moves inside an $(n-1)$-dimensional smooth and unobstructed family of deformations (the fibres of $\psi$), all deformations of $F$ stay inside $D$, and the natural evaluation map 
$ev: H^0(F,N_{F/X})\otimes
\cO_F\to TX|_F$ must take its values in $T_{D\vert F}$. 
 
  Assume the leaves of $\cF$ are not the fibres of $a_F$. We can then choose a $1$-form $u$ on $Alb(F)$ such that $v=a_F^*(u)$ does not vanish on $\cF$ at the generic point $z$ of $F$. The vanishing hyperplane of $v_z$ in $TF_z$ is however $\sigma$-dual to a vector $t_z\in TX_z$, unique and a nonzero modulo $TF_z$, which corresponds to the $1$-form $v_z$ under the isomorphism $(N_F)_z\cong (\Omega^1_F)_z$ induced by $\sigma$ on the Lagrangian $F$. Since $v$ does not vanish on $\cF_z$ by assumption, $t_z\notin (T_D)_z$, which contradicts the fact that all first-order infinitesimal deformations of $F$ are contained in $D$.
\end{proof}

\begin{corollary}\label{csa} Assume that $X$ is an irreducible hyperk\"ahler manifold of 
dimension $2n$, and $D\subset X$ a smooth reduced and irreducible divisor. Assume that $K_D$ is semi-ample. Then $\cO_X(D)$ is semi-ample.
\end{corollary}

\begin{proof} If the Beauville-Bogomolov square $q(D,D)$ is positive, then $D$ is big, $X$ is projective and the statement
follows from Kawamata base point freeness theorem. So the interesting case is when $D$ is 
Beauville-Bogomolov isotropic. We have $K_D=\cO_X(D)|_D$. If $K_D$ is semi-ample, its Kodaira
dimension is equal to $\nu(K_D)=n-1$ (lemma \ref{nu}) and the Iitaka fibration $\psi$ is regular. The 
relative dimension
of $\psi$ is equal to $n$. In fact $q(D,D)=0$ implies that $\psi$ is lagrangian in the same way
as in \cite{M} (using that $K_D=\cO_D(D)$ and that a suitable positive multiple $m.F$ of the fibre $F$ is $\psi^*(H^{n-1})$ for some very ample line bundle $H$ on $S$).  By proposition \ref{alb},
we have two possibilities: either $F$ is a torus, and then the fibration $\psi$ extends near 
$D$, since $F$ must deform in an $n$-dimensional family; or $F$ is of Albanese dimension
$n-1$ and the characteristic foliation on $D$ is algebraic.

In the first case we conclude by \cite{GLR}, \cite{HW} and \cite{M'}. In the second case,
we notice that since $F$ has numerically trivial canonical bundle, the fibers of the
characteristic foliation, which by proposition \ref{alb} are tangent to $F$, must be
elliptic curves by adjunction formula. Therefore the characteristic foliation is
isotrivial, and corollary \ref{liso} together with the proof of theorem \ref{irreducibleHK} imply 
that this is impossible unless in the case $n=1$, which is well-known.

\end{proof}

Recall that the Lagrangian conjecture affirms that a non-zero nef Beauville-Bogomolov isotropic 
divisor is semiample (and thus there is a lagrangian fibration associated to some multiple of
such a divisor). Corollary \ref{csa} shows that the Lagrangian conjecture is true for an effective
smooth divisor on a holomorphic symplectic manifold of dimension $2n$, if the Abundance conjecture
holds in dimension $2n-1$. Since the Abundance conjecture is known in dimension $3$,
we have the following:

\begin{corollary}\label{dim4} Let $X$ be an irreducible hyperk\"ahler manifold of dimension $4$, and $D$ a nef divisor on $X$. Assume that $D$ is effective
and smooth. Then $\cO_X(D)$ is semi-ample. 
\end{corollary} 

Notice that if $dim(X)=4$, we can use \cite{A} instead of \cite{GLR} and \cite{HW}, and \cite{AC} 
instead of \cite{M'}, so that the proof becomes more elementary in this case.

\medskip

{\bf Acknowledgements:} We are grateful to Jorge Pereira who suggested to look at $d-1$-forms rather than at 
$2$-forms, and to Michael McQuillan who asked a question which led us to a simplification of the 
original argument.

The first author's research was carried out within the 
National Research University Higher
School of Economics
Academic Fund Program for 2015-2016, research grant No. 15-01-0118.

\end{document}